\documentclass[a4paper, 11pt]{article}

\usepackage[arxiv]{optional}
\usepackage[english]{babel}
\usepackage[T1]{fontenc}
\usepackage{amsmath}
\opt{arxiv}{%
\usepackage{amsthm}
}%
\usepackage{amsfonts}
\usepackage{amssymb}
\usepackage[final]{graphicx}
\usepackage{paralist}
\opt{draft}{%
\usepackage{refcheck}
}

\opt{arxiv}{%
\theoremstyle{plain}
\newtheorem{thm}{Theorem}[section]
\newtheorem{prop}[thm]{Proposition}
\newtheorem*{prop*}{Proposition}
\newtheorem{lem}[thm]{Lemma}
\newtheorem*{lem*}{Lemma}
\newtheorem{cor}[thm]{Corollary}
\newtheorem*{cor*}{Corollary}

\newtheorem*{example*}{Example}
\newtheorem{conject}[thm]{Conjecture}
\newtheorem*{conject*}{Conjecture}

\theoremstyle{definition}
\newtheorem{defn}[thm]{Definition}
\newtheorem*{defn*}{Definition}
\newtheorem{rem}[thm]{Remark}
\newtheorem*{rem*}{Remark}
}%
\opt{springer}{%

}%

\newcommand{\eps}{\varepsilon}

\DeclareMathOperator{\ind}{ind}

\opt{springer}{%
\newcommand{\eop}{\qed}
}
\opt{arxiv}{%
\newcommand{\eop}{}
}

\newcommand{\abs}[1]{\vert #1 \vert}

\newcommand{\C}{\ensuremath{\mathbb{C}}}

\newcommand{\defby}{\mathrel{\mathop:}=}

\newcommand{\bydef}{=\mathrel{\mathop:}}

\newcommand{\coloneq}{\mathrel{\mathop:}=}

\newcommand{\conj}[1]{\overline{#1}}

\title{A Note on the Maximum Number of Zeros of $r(z) - \conj{z}$}

\author{Robert Luce, Olivier S\`{e}te, J\"org Liesen\footnote{TU
Berlin, Institute of Mathematics, MA 4-5, Stra{\ss}e des 17. Juni 136,
10623 Berlin, Germany (\texttt{\{luce,sete,liesen\}@math.tu-berlin.de})}}

\begin{document}
\maketitle

\begin{abstract}
An important theorem of Khavinson \& Neumann (Proc. Amer. Math. Soc.
134(4), 2006) states that the complex harmonic function $r(z) -
\conj{z}$, where $r$ is a rational function of degree $n \geq 2$, has
at most $5 (n - 1)$ zeros.  In this note we resolve a slight
inaccuracy in their proof and in addition we show that for certain
functions of the form $r(z) - \conj{z}$ no more than $5 (n - 1) - 1$
zeros can occur.  Moreover, we show that $r(z) - \conj{z}$ is regular,
if it has the maximal number of zeros.

\end{abstract}

\section{Introduction}

Let $r = \frac{p}{q}$ be a complex rational function of degree 
\begin{equation*}
    n = \deg(r) \defby \max \{ \deg(p), \deg(q) \}.
\end{equation*}
Here and in the sequel the polynomials $p$ and $q$ are always assumed to 
be coprime.  We then say that the rational harmonic function
\begin{equation}
\label{eqn:rat_harm}
    f(z) \defby r(z) - \conj{z}
\end{equation}
is of degree $n$, too.  Such functions have an interesting application
in \emph{gravitional microlensing}; see the introductory overview
article of Khavinson \& Neumann~\cite{KhavinsonNeumann2008}.  They
also play a role in the matrix theory problem of expressing certain
adjoints of diagonalizable matrices as rational functions of the
matrix~\cite{Liesen2007}.

An important theorem of Khavinson \& Neumann~\cite[Theorem~1]{KhavinsonNeumann2006}
states that a rational harmonic function~\eqref{eqn:rat_harm} of degree $n \ge 2$
has at most $5(n-1)$ zeros. In this note we give an alternative proof 
of their result (the differences to the original proof are discussed
in Remark~\ref{rem:difference}). Moreover,
we show that a slightly better bound can be given if one takes into account 
the individual degrees of the numerator and denominator polynomials.  
In order to state our main result, 
we recall that a zero $z_0$ of $f$ is called \emph{sense-preserving}
if $\abs{r'(z_0)} > 1$, \emph{sense-reversing} if $\abs{r'(z_0)} < 1$,
and \emph{singular} if $\abs{r'(z_0)} = 1$;
see~\cite{SeteLuceLiesen2014}.

\begin{thm}
\label{thm:new_bound}
A rational harmonic function $f(z) = r(z)- \conj{z}$ of degree $n \ge 2$ has at 
most $3(n-1)$ sense-preserving zeros, and at most $2(n-1)$ sense-reversing or 
singular zeros.
Moreover, if $r = \tfrac{p}{q}$ with $\deg(p) > \deg(q)$, then $f$ 
has at most $5(n-1)-1$ zeros.
\end{thm}

The first part of this theorem was already stated in~\cite[Theorem~1
and Proposition~1]{KhavinsonNeumann2006} (see
also~\cite[Appendix~B]{AnEvans2006}, where several extensions to this
bound are presented). Our proof in the next section employs similar
techniques as the one in~\cite{KhavinsonNeumann2006}, but it avoids a
subtle inaccuracy in the argument, which we will explain next.

If $f(z) = r(z) - \conj{z}$ has no singular zero, then $f$ as well as
$r$ are called \emph{regular}.  In the proof of the Main Lemma
in~\cite{KhavinsonNeumann2006}, part (2), it is implicitly assumed
that if $f(z) = r(z) - \conj{z}$ is regular, then the function
\begin{equation*}
    F(w) = \tfrac{1}{r(\frac{1}{w})} - \conj{w}
\end{equation*}
is regular as well.  However, this implication is in general not correct.
For example, consider the rational harmonic function $f(z) = z + \tfrac{1}{z} - 
\conj{z}$.  Clearly, $0$ is not a zero of $f$, so that we have
\begin{equation*}
    f(z) = 0 \; \Leftrightarrow \; z^2 + 1 = \abs{z}^2,
\end{equation*}
and hence $f$ has (only) the two zeros $\pm \tfrac{i}{\sqrt{2}}$.
Since $\abs{r'(\pm \tfrac{i}{\sqrt{2}})} = 3 > 1$, the function $f$ is regular. 
Now consider
\begin{equation}
\label{eqn:defR}
    F(w)
    = \tfrac{1}{r(\frac{1}{w})} - \conj{w}
    = \tfrac{ w }{ 1 + w^2 } - \conj{w}
    \bydef R(w) - \conj{w}.
\end{equation}
Then $F(0) = 0$, and $\abs{R'(0)} = 1$ shows that $0$ is a singular zero of $F$.

In Section~\ref{sec:new_proof} we give a new proof of 
Theorem~\ref{thm:new_bound}.  In Section~\ref{sec:extremal} we further show 
that $r(z)-\conj{z}$ has no singular zeros, if it has the maximal number of 
zeros as stated in Theorem~\ref{thm:new_bound}.

\section{Proof of Theorem~\ref{thm:new_bound}}
\label{sec:new_proof}

In order to prove Theorem~\ref{thm:new_bound} we need some preliminary
results.  First note that the function $R$ defined in~\eqref{eqn:defR}
can be written as
\begin{equation}
\label{eqn:co-conj}
R(w)=\conj{T}\circ r \circ T^{-1}(w), 
\end{equation}
where $w = T(z) = \tfrac{1}{z}$ is a M\"obius transformation. More
generally, we say that for a rational function $r(z)$ and any given
M\"obius transformation $T(z)$,  a function $R(w)$ of
the form~\eqref{eqn:co-conj} is a \emph{co-conjugate} of $r(z)$.  Here
$\conj{T}(z)$ denotes the M\"obius transformation obtained from $T(z)$
by conjugating all the coefficients.  Co-conjugates maintain the
number and sense of zeros of $r(z) - \conj{z}$, as we show next.

\begin{prop}
\label{prop:co-conjugate}
Let $r(z)$ be rational and of degree $n \ge 1$, and let $T(z) =
\frac{az + b}{cz + d}$ be a M\"{o}bius transformation.  Then $R(w) =
\conj{T} \circ r \circ T^{-1}(w)$ is a rational function of degree $n$
and we have:
\begin{compactenum}
\item $r(z) = \conj{z}$ if and only if $R(w) = \conj{w}$, for
all $z \in \C$ with $w = T(z) \neq \infty$.  In that case, if $r(z) =
\conj{z}$, we have $\abs{r'(z)} = \abs{R'(w)}$.

\item Writing
$r = \tfrac{p}{q}$ with $p(z) = \sum_{k=0}^n p_k z^k$ and 
$q(z) = \sum_{k=0}^n q_k z^k$, $R$ has the representation
\begin{equation}
\label{eqn:co-conjugate}
    R(w) = \frac{ \sum_{k=0}^n ( \conj{a} p_k + \conj{b} q_k)
    (dw-b)^k (a-cw)^{n-k} }{ \sum_{k=0}^n (\conj{c} p_k + \overline{d} q_k )
    (dw-b)^k (a-cw)^{n-k} }.
\end{equation}
\end{compactenum}
\end{prop}

\begin{proof}
The degree of $R$ can be seen from the degree formula
$\deg(r \circ s) = \deg(r) \deg(s)$ for non-constant rational functions;
see~\cite[p.~32]{Beardon1991}.
The first claim can be seen from the computations
\begin{equation*}
    r(z) = \conj{z}
    \Leftrightarrow (\conj{T} \circ r)(z) = \conj{T}( \conj{z}) = \conj{T(z)}
    \Leftrightarrow R(w) = \conj{w},
\end{equation*}
and
\begin{equation*}
    R'(w)
    = \conj{T}' (r(z)) r'(z) (T^{-1})' (w)
    = \conj{T}' (\conj{z}) r'(z) \tfrac{1}{T'(z)}
    = \tfrac{\conj{T'(z)}}{T'(z)} r'(z).
\end{equation*}
For the second claim, note that $T^{-1}(w) = \tfrac{dw-b}{a-cw}$, so that
we have
\begin{equation*}
    r(T^{-1}(w))
    = \frac{ \sum_{k=0}^n p_k (dw-b)^k (a-cw)^{n-k} }{ \sum_{k=0}^n
        q_k (dw-b)^k (a-cw)^{n-k} },
\end{equation*}
from which we see that $R(w) = \conj{T}(r(T^{-1}(w)))$ has the
form~\eqref{eqn:co-conjugate}.
\eop
\end{proof}

In our proof of Theorem~\ref{thm:new_bound} we also need the winding
of a complex function along a curve, and indices of zeros and poles of
harmonic functions (sometimes called order, or multiplicity).  Here we
only give the most relevant definitions. A
compact summary of these concepts is given in~\cite[Section
2]{SeteLuceLiesen2014}, see also~\cite{Balk1991} and
\cite[p.~29]{Sheil-Small2002} (where the winding is called ``degree'').

Let $\Gamma$ be a rectifiable curve with
parametrization $\gamma : [a, b] \to \Gamma$.  Let $f : \Gamma \to \C$
be a continuous function with no zeros on $\Gamma$.  Let $\arg f(z)$
denote a continuous branch of the argument of $f$ on $\Gamma$.  The
\emph{winding} (or \emph{rotation}) \emph{of $f(z)$ on the curve
$\Gamma$} is defined as
\begin{equation*}
V(f; \Gamma)
\coloneq \tfrac{1}{2 \pi} ( \arg f(\gamma(b)) - \arg f(\gamma(a)) )
\end{equation*}
The winding is independent of the choice of the branch of $\arg f(z)$.
Let $z_0$ be a zero or pole of $f(z) = r(z) - \conj{z}$. Denote by
$\Gamma$ a circle around $z_0$ not containing any further zeros or
poles of $f$.  Then the \emph{Poincar\'e index} of $f$ at $z_0$ is
defined as
\begin{equation*}
    \ind(z_0; f) \defby V(f; \Gamma).
\end{equation*}
The Poincar\'e index is independent of the choice of $\Gamma$.

Moreover, we will use the following results
in our proof.

\begin{prop}[{\cite[Proposition~2.7]{SeteLuceLiesen2014}}]
\label{prop:regular_indices}
Let $f(z) = r(z) - \conj{z}$ be a rational harmonic function with $\deg(r)
\geq 2$.  The indices of $f$ at $z_0$ can be summarized as follows:
\begin{compactenum}
\item If $z_0$ is a sense-preserving zero of $f$, then $\ind(z_0;f) = 1$.
\item If $z_0$ is a sense-reversing zero of $f$, then $\ind(z_0;f) = -1$.
\item If $z_0$ is a pole of $r$ of order $m$, then $\ind(z_0;f) = -m$.
\end{compactenum}
\end{prop}

\begin{prop}[{\cite[Proposition~1]{KhavinsonNeumann2006}}]
\label{prop:bound_sr_neutral_zeros}
A rational harmonic function $f(z) = r(z) - \conj{z}$ of degree $n \ge 2$ 
has at most $2(n-1)$ sense-reversing or singular zeros. 
\end{prop}

\begin{lem}[{\cite[Lemma]{KhavinsonNeumann2006}}]
\label{lem:r-c_regular}
If $r$ is rational and of degree at least $2$, then the set of complex 
numbers $c$ for which $r-c$ is regular, is open and dense in $\C$.
\end{lem}

A useful application of the preceeding ``density lemma'' emerges when
combined with the continuity of the non-singular zeros of harmonic
functions.  In the following we call $f$ sense-preserving on an open
subset $U$, if $\abs{r'(z)}
> 1$ for all $z \in U$ (similarly for sense-reversing).

\begin{lem}
\label{lem:rootmigration}
Let $f(z) = r(z) - \conj{z}$ with $\deg(r) \ge 2$.  Then for every
sufficiently
small $\eps > 0$ there exists $\delta > 0$ such such that for $\abs{c}
<
\delta$ holds: For every sense-preserving zero $z_0$ of $f$, the
perturbed
function $f - c$ has exactly one zero $z_0'$ in $\{ z : \abs{z-z_0} <
\eps \}$,
which is again sense-preserving.
The same applies to sense-reversing zeros.

In particular the function $f-c$ has at least as many sense-preserving (and
sense-reversing)
zeros as $f$.
\end{lem}

\begin{proof}
Let $\Omega_+ = \{ z : \abs{r'(z)} > 1 \}$ be the set where $f$ is
sense-preserving.  Denote the sense-preserving zeros of $f$ by $z_1,
\ldots,
z_{n_+}$.  Let $\eps > 0$ be sufficiently small such that
\begin{compactenum}
\item all disks $\{ z : \abs{z-z_j} \leq \eps \}$ are mutually
disjoint and
contained in $\Omega_+$,

\item $f$ has no zero or pole in each $\{ z : 0 < \abs{z-z_j} \leq
\eps \}$.
(This is possible, since the zeros and poles of $f$ are isolated.)
\end{compactenum}
Fix $j \in \{ 1, \ldots, n_+ \}$.  Set $\gamma_j = \{ z : \abs{z-z_j}
= \eps
\}$ and let $\delta_j = \min_{z \in \gamma_j} \abs{f(z)} > 0$.  Then,
for any $\abs{c} < \delta_j$ we have
\begin{equation*}
    \abs{ f - (f-c) }
    = \abs{c} < \delta_j
    \leq \abs{f} + \abs{f-c} \quad \text{ on } \gamma_j.
\end{equation*}
Rouch\'e's theorem shows $V(f-c; \gamma_j) = V(f; \gamma_j) = +1$, so
$f-c$
(again sense-preserving on $\Omega_+$) has exactly one
sense-preserving zero
interior to $\gamma_j$ (Proposition~\ref{prop:regular_indices}
combined with the argument
principle~\cite{DurenHengartnerLaugesen1996}).
The same applies to sense-reversing zeros (by considering the set $\Omega_-
= \{ z :
\abs{r'(z)} < 1 \}$).  Taking $\delta$ as the minimum of all
$\delta_j$
completes the proof.
\eop
\end{proof}

\opt{arxiv}{
\begin{proof}[{\bfseries Proof of Theorem~\ref{thm:new_bound}}]
}
\opt{springer}{
\begin{proof}[of Theorem~\ref{thm:new_bound}]
}
Let us denote
\begin{equation*}
    r = \tfrac{p}{q}, \quad
    p(z) = \sum_{k=0}^n p_k z^k, \quad
    q(z) = \sum_{k=0}^n q_k z^k,
\end{equation*}
and let $n_+, n_0, n_-$ be the number of sense-preserving, singular,
sense-reversing zeros of $f$, respectively.
Sometimes we make the dependence on $f$ explicit by writing $n_+(f)$ etc.

By Proposition~\ref{prop:bound_sr_neutral_zeros}, $n_{-,0} \coloneq n_- + n_0 
\leq 2(n-1)$.  It therefore remains to show that $n_+ \leq 3(n-1)$ and to show 
that $f$ has at most $5(n-1)-1$ zeros when $\deg(p) > \deg(q)$. We divide the 
proof in four steps. 

\textbf{Step 1:} Let $r$ be regular with $\deg(p) \leq \deg(q) = n$,
so the number of singular zeros is $n_0 = 0$.  Let $\gamma$ be a
circle containing all zeros and poles of $f$.  In this case, since $r$
is bounded for $z \to \infty$, we have
\begin{equation*}
    \abs{\conj{z} - f(z)} = \abs{r(z)} \le C <
    \abs{\conj{z}} + \abs{f(z)}, \quad z \in \gamma,
\end{equation*}
provided that $\gamma$ is sufficiently large.  Rouch\'e's
theorem~\cite[Theorem~2.3]{SeteLuceLiesen2014} implies $V(f; \gamma)
= V(\conj{z}; \gamma) = -1$. Applying the argument principle for complex-valued 
harmonic functions yields
\begin{equation*}
    -1
    = V(f; \gamma)
    = \sum_{z_j : f(z_j) = 0} \ind(z_j; f) +
        \sum_{ z_j : q(z_j) = 0} \ind(z_j; f)
    = n_+ - n_- - n,
\end{equation*}
where we used Proposition~\ref{prop:regular_indices}.  In particular, the sum 
of the orders of the poles of $f$ is equal to $\deg(q) = n$.  By 
Proposition~\ref{prop:bound_sr_neutral_zeros} we have $n_- \leq 2(n-1)$.  
Thus,
\begin{equation*}
    n_+ = n-1 + n_- \leq n-1 + 2(n-1) = 3(n-1).
\end{equation*}

\textbf{Step 2:} Let $\deg(p) \leq \deg(q) = n$.  If $r$ is regular,
we are done by Step~1, so assume that $r$ is not regular.  By
Lemma~\ref{lem:r-c_regular} there exists a sequence $c_k \in \C$ such
that $r_k(z) \coloneq r(z) - c_k$ are regular and $c_k \to 0$.  Then
$r_k$ satisfies the conditions of Step~1 and, setting $f_k(z) \coloneq
r_k(z) - \conj{z}$, we have $n_+(f_k) \leq 3(n-1)$ by Step~1.  For
sufficiently small $\abs{c_k}$,
Lemma~\ref{lem:rootmigration} shows that
the function $f_k = f - c_k$ has at
least as many sense-preserving zeros as $f$, that is, $n_+(f) \leq
n_+(f_k) \leq 3(n-1)$.

\textbf{Step 3:} Let $n = \deg(p) > \deg(q)$ and $p(0) \neq 0$.  In this
case we have $p_n \neq 0$, $p_0 \neq 0$ and $q_n = 0$. Let $w = T(z) =
\frac{1}{z}$, then
\begin{equation*}
R(w) = \conj{T} \circ r \circ T^{-1} (w) = \tfrac{1}{r(\frac{1}{w})}
= \frac{\sum_{k=0}^n q_k w^{n-k}}{\sum_{k=0}^n p_k w^{n-k}},
\end{equation*}
which can be seen from~\eqref{eqn:co-conjugate}.  Since $p_0 \neq 0$,
we see that $F(w) = R(w) - \conj{w}$ satisfies the conditions in
Step~2.  Thus, $n_+(F) \le 3(n - 1)$ and $n_{-,0}(F) \le 2(n-1)$.

Since $f(0) = \frac{p(0)}{q(0)} \neq 0$, every zero $z_j$ of $f$
gives rise to a zero $w_j = T(z_j)$ of $F$, and every zero $0 \neq
w_j$ of $F$ corresponds to a zero $z_j = \frac{1}{w_j}$ of $f$;
see Proposition~\ref{prop:co-conjugate}.  Since the senses of the
zeros are preserved under the co-conjugation with $T$, we find
\begin{equation*}
n_+(f) \leq n_+(F) \leq 3(n-1) \quad \text{ and } \quad n_{-,0}(f) \leq 
n_{-,0}(F) \leq 2(n-1).
\end{equation*}
Notice that $F(0) = 0$, since $q_n = 0$.  This zero of $F$ has no
corresponding zero of $f$, so that $f$ has at most $5(n-1)-1$ zeros.

\textbf{Step 4:} Let $n = \deg(p) > \deg(q)$ and $p(0) = 0$.  In that
case we have  $p_n \neq 0$, $q_n = 0$ and $p_0 = 0$.  Let $b \in \C$
satisfy $r(b) \neq \conj{b}$.  With the M\"{o}bius transformation
$T(z) = z-b$ we consider 
\begin{equation*}
    R(w)
    = \conj{T} \circ r \circ T^{-1}(w)
    = \frac{ \sum_{k=0}^n ( p_k + -\conj{b} q_k) (w+b)^k }
        { \sum_{k=0}^n q_k (w+b)^k };
\end{equation*}
see Proposition~\ref{prop:co-conjugate}.
The coefficient of $w^n$ in the numerator of $R$ is $p_n - \conj{b}
q_n = p_n \neq 0$, and in the denominator it is $q_n = 0$.  Further,
the constant term of the numerator of $R$ is
\begin{equation*}
    \sum_{k=0}^n ( p_k - \conj{b} q_k )b^k
    = p(b) - \conj{b} q(b) \neq 0,
\end{equation*}
since $r(b) \neq \conj{b}$.  Thus $F(w) \coloneq R(w) - \conj{w}$
satisfies the conditions in Step~3, so that
\begin{equation*}
n_{-,0}(F) \leq 2(n-1) \quad \text{and} \quad n_+(F) \leq 3(n-1),
\end{equation*}
and $F$ has at most $5(n-1)-1$ zeros.

Proposition~\ref{prop:co-conjugate} implies that $r(z) = 
\conj{z}$ if and only if $R(w) = \conj{w}$, where $w = T(z)$.  Thus
$f$ and $F$ have the same number of zeros, and all corresponding
zeros have the same sense (or are singular).  Hence $n_{-,0}(f) =
n_{-,0}(F) \leq 2(n-1)$ and $n_+(f) = n_+(F) \leq 3(n-1)$, and the
total number of zeros of $f$ is bounded by $5(n-1)-1$.
\eop
\end{proof}

\begin{rem}
Let $f(z) = \frac{p(z)}{q(z)} - \conj{z}$ with $\deg(p) > \deg(q)$, so
that $f$ has at most $5(n-1) - 1$ zeros.  Then the point $\infty$ in
$\hat{\C}$ can be regarded as the ``missing solution'' to $r(z) =
\conj{z}$.  However, the point infinity can \emph{not} be a zero of
the function $r(z) - \conj{z}$,
see~\cite[p.~1078]{KhavinsonNeumann2006}.
\end{rem}

\begin{rem}
In Step~3 in the above proof, one can infer the type of the zero $w =
0$ of $F$. In this step $p_n \neq 0$ and $q_n = 0$.
We compute
\begin{equation*}
    R'(w)
    = \tfrac{ - r'(\frac{1}{w}) \frac{-1}{w^2} }{ r(\frac{1}{w})^2 }
    = r'(z) \tfrac{ z^2 }{ r(z)^2 }
    = \tfrac{ ( p'(z) q(z) - p(z) q'(z) ) z^2 }{p(z)^2}.
\end{equation*}
Note that $z^{2n}$ is the highest power of $z$ that may occur in both numerator 
and denominator.  The coefficient of $z^{2n}$ in the denominator is $p_n^2$, and
in the numerator it is
\begin{equation*}
n p_n q_{n-1} - p_n q_{n-1} (n-1) = p_n q_{n-1},
\end{equation*}
which yields
\begin{equation*}
R'(0) = \lim_{w \to 0} R'(w)
=\lim_{z \to \infty} \tfrac{ ( p'(z) q(z) - p(z) q'(z) ) z^2 }{p(z)^2}
= \tfrac{ q_{n-1} }{p_n}.
\end{equation*}
This shows that $w = 0$ may be a sense-preserving, sense-reversing or
singular zero of $F$.
\end{rem}

\begin{rem}
\label{rem:difference}

Let us briefly discuss how our proof of Theorem~\ref{thm:new_bound}
differs from the original proof of Khavinson \& Neumann
in~\cite{KhavinsonNeumann2006}. A major ingredient in both proofs is
Proposition~\ref{prop:bound_sr_neutral_zeros}, due to Khavinson \&
Neumann, which bounds the number of sense-reversing and singular
zeros. Because of this result it only remains to bound the number of
sense-preserving zeros.  Here the two main technical challenges are
(i) dealing with singular zeros, and (ii) the slightly different
behavior of rational functions $f(z) = \frac{p(z)}{q(z)} - \conj{z}$
with $\deg(p) \le \deg(q)$ and $\deg(p)
> \deg(q)$.
The main difference between the two proofs is the order in which (i)
and (ii) are handled.  While Khavinson \& Neumann first resolve (ii)
under the assumption that all zeros are regular, and then apply the
density lemma (Lemma~\ref{lem:r-c_regular}) to resolve (i), our proof
first treats the case $\deg(p) \le \deg(q)$, using the density lemma
(steps $1$ and $2$), and then transfers the result to the other case
using Proposition~\ref{prop:co-conjugate} (steps $3$ and $4$).  By
this order we avoid a transformation of variables
which may introduce singular zeros at a stage of the proof where this
case is not covered.
\end{rem}

The bounds in Theorem~\ref{thm:new_bound} are sharp.  If $f(z) = r(z)
- \conj{z}$ of degree $n \geq 2$ attains the maximal number of
  $5(n-1)$ zeros, we call $f$ and $r$ \emph{extremal}.  Examples of
extremal functions were constructed by Rhie~\cite{Rhie2003}.  She
considered the function
\begin{equation}
f(z) = \tfrac{z^{n-1}}{z^n - a^n} - \conj{z} 
\label{eqn:mpw_function}
\end{equation}
which is extremal for degree $n = 2, 3$ for a special value of $a \in (0,1)$, 
and the function
\begin{equation}
f(z) = (1-\eps) \tfrac{z^{n-1}}{z^n - a^n} + \tfrac{\eps}{z} - \conj{z}
= \tfrac{z^n - \eps a^n}{ z^{n+1} - a^n z }, 
\label{eqn:rhie_function}
\end{equation}
of degree $n+1$ which is extremal for $n \geq 3$, provided that $\eps$
is sufficiently small.  See~\cite{LuceSeteLiesen2014} for a rigorous
analysis of admissible parameters $a$ and $\eps$ such that these
functions are indeed extremal.  Note that the rational function
in~\eqref{eqn:rhie_function} is a convex combination of the rational
function in~\eqref{eqn:mpw_function} and a pole located at a zero
of~\eqref{eqn:mpw_function}.  This general construction principle for
extremal functions has been studied in detail
in~\cite{SeteLuceLiesen2014}.

A phase portrait (see~\cite{WegertSemmler2011,Wegert2012}
and~\cite[Section 4]{SeteLuceLiesen2014}) of an extremal function of
the form~\eqref{eqn:rhie_function} with $n = 4$, and $\eps = 0.04$ is
shown in Figure~\ref{fig:bounds_sharp} (left).

\begin{figure}
\includegraphics[width=.49\textwidth]{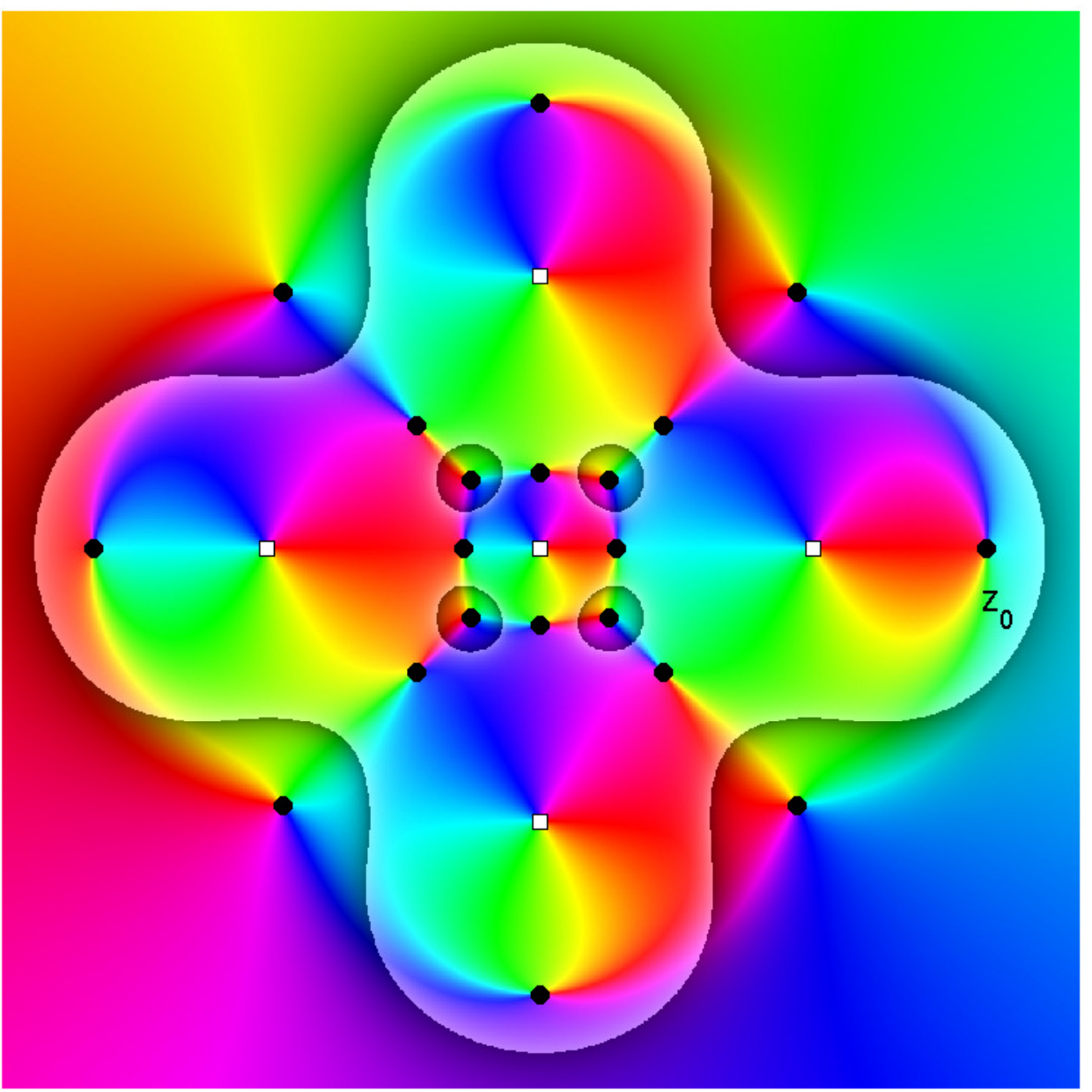} \hfill
\includegraphics[width=.49\textwidth]{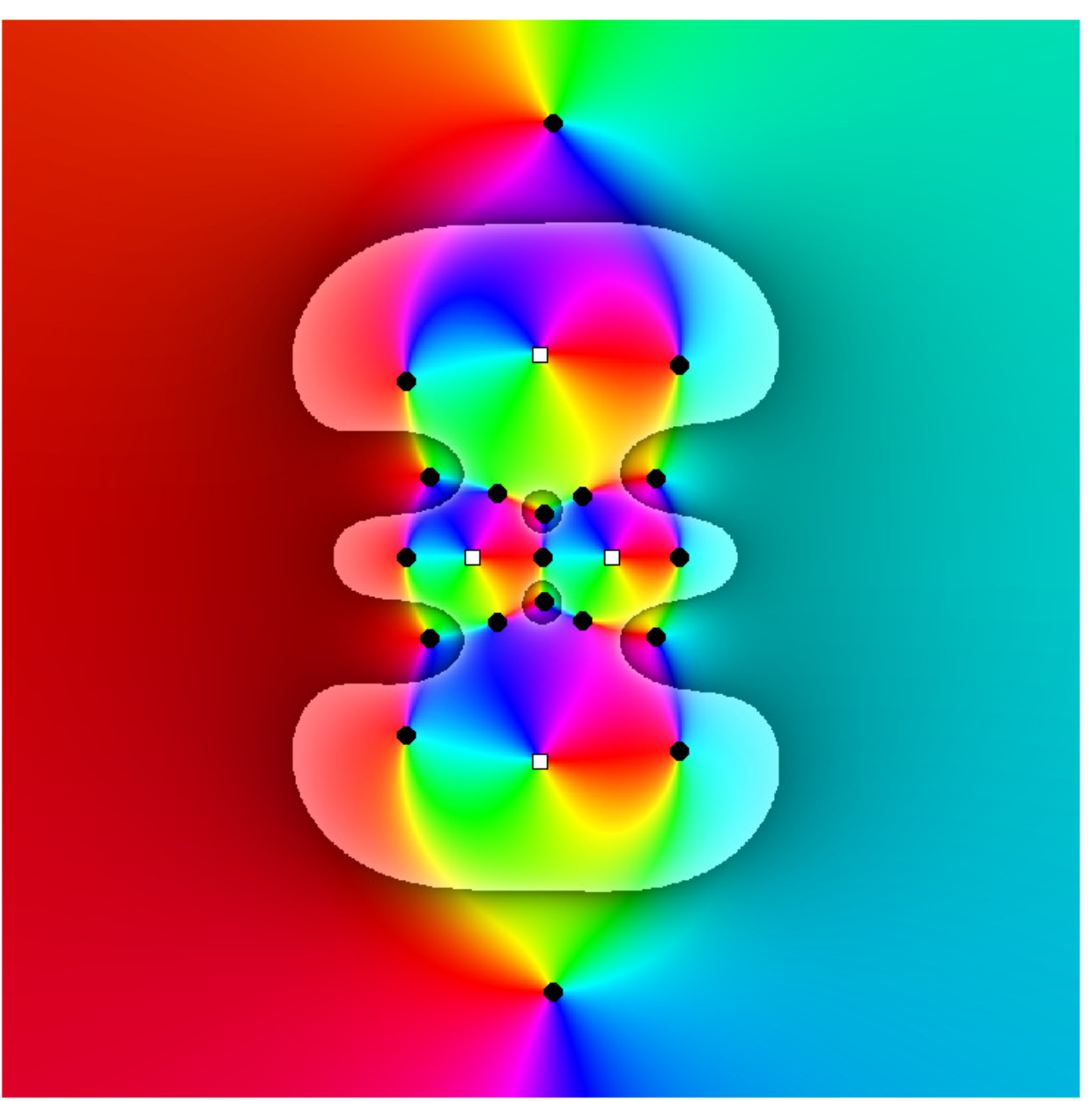}
\caption{Phase portrais of~\eqref{eqn:rhie_function} (left), and
of~\eqref{eqn:co-conj_rhie} (right).  Each image shows a part of
the domain of the corresponding function.  By indentifying the unit
circle with the standard HSV color wheel, each point in the domain is
colored according to the phase $e^{i \arg{f(z)}}$ of the function
value at that point (see~\cite{Wegert2012}).  The brightened regions
indicate the parts of the domain where the function is
sense-preserving.  Black disks denote zeros, white squares poles.
Both functions are of degree five, and have $20$ and $19$ zeros,
respectively,  which is the maximum possible number in each case.}
\label{fig:bounds_sharp}
\end{figure}

We show that for rational functions with $\deg(p) > \deg(q)$ the new bound from 
Theorem~\ref{thm:new_bound} is also sharp.  Let $r$ be the any of the extremal 
rational functions from~\eqref{eqn:mpw_function} and~\eqref{eqn:rhie_function}. 
Let $z_0$ be any zero of $f(z) = r(z) - \conj{z}$ and consider the co-conjugate 
of $r$ with $w = T(z) = \tfrac{1}{z-z_0}$:
\begin{equation*}
R(w) = \conj{T} \circ r \circ T^{-1}(w) = \tfrac{1}{ r(\frac{1}{w} + z_0) - 
\conj{z}_0 }. 
\end{equation*}
From Proposition~\ref{prop:co-conjugate} it is easy to see that the numerator 
of $R$ has degree $\deg(r)$ and the denominator has degree (at most) 
$\deg(r)-1$.  Further the zeros of
\begin{equation}
F(w) = R(w) - \conj{w} \label{eqn:co-conj_rhie}
\end{equation}
are exactly the images of the zeros ($\neq z_0$) of $f(z) = r(z) - \conj{z}$, 
so that $F$ has $5(\deg(R)-1)-1$ zeros.  Figure~\ref{fig:bounds_sharp} (right) 
illustrates this construction for $n = 4$, where $z_0$ is the
rightmost zero of $f$ in the left phase portrait.

\section{Extremal Rational Harmonic Functions are Regular}
\label{sec:extremal}

In this section we will show that functions $f(z) = r(z) - \conj{z}$
that attain the maximum number of zeros as stated in
Theorem~\ref{thm:new_bound} are, surprisingly, guaranteed
to be regular.

\begin{thm}
Let $r = \frac{p}{q}$ be a rational function of degree $n \ge 2$ and set $f(z) =
r(z) - \conj{z}$. If
\begin{compactenum}[(i)]
\item $f$ has $5(n-1)$ zeros, or\label{item:good_case}
\item $\deg(p) > \deg(q)$ and $f$ has $5(n-1) - 1$ zeros,\label{item:bad_case}
\end{compactenum}
then none of the zeros are singular.
\end{thm}

\begin{proof}
\textit{(\ref{item:good_case})} Let $\Omega_+ \coloneq \{ z : \abs{r'(z)} > 1 \}$ be the set where $f$ is 
sense-preserving.  Denote by $n_+$ the number of zeros of $f$ in $\Omega_+$
and by $n_{-,0}$ the number of zeros of $f$ in $\{ z : \abs{r'(z)} \leq 1 \}$.
Since $f$ has $5(n-1)$ zeros, Theorem~\ref{thm:new_bound} implies
\begin{equation*}
n_+ = 3(n-1), \quad n_{-,0} = 2(n-1).
\end{equation*}

Suppose $f$ has a singular zero $z_0$.  
Let $z_1, \ldots, z_{n_+}$ be the $n_+ = 3(n-1)$ zeros of $f$ in $\Omega_+$.
Let $\eps > 0$ be such that the disks 
$\{ z : \abs{z-z_j} \leq \eps \}$ do not intersect for $0 \leq j \leq n_+$, and 
are contained in $\Omega_+$ for $1 \leq j \leq n_+$.
By Lemma~\ref{lem:rootmigration} applied to $f$ on $\Omega_+$ there exists 
$\delta > 0$ such that for all $\abs{c} < \delta$ the function $f - c$ 
has exactly one zero in each $\eps$-disk $D_\eps(z_j) = \{ z : \abs{z-z_j} < 
\eps \}$, $1 \leq j \leq n_+$.

Now, since $f(z_0) = 0$ and $f$ is continuous near $z_0$, there exists
$0 < \eta \leq \eps$ such that $\abs{f(z)} < \delta$ in $D_\eta(z_0) = \{ z :
\abs{z-z_0} < \eta \}$.
Further, there exists $\zeta \in D_\eta(z_0) \cap \Omega_+$.  Indeed, assume 
the contrary, then $\abs{r'(z)} \leq 1$ in $D_\eta(z_0)$ and $\abs{r'(z_0)} = 1$,
which implies that $r'$ is constant by the maximum modulus theorem, a
contradiction to $\deg(r) \geq 2$.

Finally, consider the function $F(z) \coloneq f(z) - f(\zeta)$.  Since
$\abs{f(\zeta)} < \delta$, $F$ has exactly one zero in each disk 
$D_\eps(z_j)$, $1 \leq j \leq n_+$, and further $F(\zeta) = 0$.  Thus $F$ has
$n_+ + 1 = 3(n-1) + 1$ distinct sense-preserving zeros in $\Omega_+$, in 
contradiction to Theorem~\ref{thm:new_bound}.  Therefore $f$ has no singular 
zeros.

\textit{(\ref{item:bad_case})} We reduce this case to the previous
one.  Let $b \in \C$ such that $r(b) \neq \conj{b}$, and define the
M\"obius transformation $w = T(z) = \frac{1}{z - b}$.  Consider the
co-conjugate $R = \conj{T} \circ r \circ T^{-1}$ and $F(w) = R(w) -
\conj{w}$.  By Proposition~\ref{prop:co-conjugate}, all $5(n-1) - 1$
zeros of $f$ transform to zeros of $F$ with $\abs{r'(z)} =
\abs{R'(w)}$, so that the senses are preserved.  Note that none of the
zeros of $f$ is mapped to $0$ under $T$. However,
\eqref{eqn:co-conjugate} and $n = \deg(p) > \deg(q)$ imply $R(0) =
\frac{q_n}{p_n} = 0$, so that we have $F(0) = 0$.  Thus $F$ has a
total number of $5(n-1)$ zeros, none of which is singular by the first
part.  Hence none of the zeros of $f$ are singular.  \eop
\end{proof}

\paragraph*{Ackowledgements}

We would like to thank Dmitry Khavinson for pointing out
reference~\cite{AnEvans2006} to us.  We are grateful to the
anonymous referee for the careful reading of the manuscript, and
many valuable remarks.

Robert Luce's work is supported by Deutsche For\-schungs\-gemeinschaft,
Cluster of Excellence ``UniCat''.

\bibliographystyle{plain}
\bibliography{note_kn}

\end{document}